\documentclass[12pt]{article}

\usepackage{tikz}
\usepackage{mathtools}
\usetikzlibrary{positioning}
\usetikzlibrary{matrix,arrows,decorations.pathmorphing, shapes.geometric,patterns}
\tikzset{>=latex}
\usetikzlibrary{arrows}
\usepackage{subcaption}
\usepackage{tkz-euclide}
\usepackage{sidecap}

\usepackage{amssymb}
\usepackage{amsthm}
\usepackage{amsmath}
\usepackage{graphicx}
\usepackage{fullpage}
\usepackage{color}
\usepackage{enumerate}
 \numberwithin{equation}{section}
\usepackage{booktabs}
\allowdisplaybreaks

\usepackage[pdftex]{hyperref}
 \usepackage{mathtools}
 \mathtoolsset{showonlyrefs}
\usepackage{esint}

\theoremstyle{plain}
\newtheorem{thm}{Theorem}[section]
\newtheorem{cor}[thm]{Corollary}

\newtheorem{lem}[thm]{Lemma}
\newtheorem{prop}[thm]{Proposition}

\theoremstyle{definition}
\newtheorem{defn}[thm]{Definition}

\theoremstyle{remark}

\newcommand{\N}{\mathbb{N}}
\newcommand{\R}{\mathbb{R}}

\newcommand{\bp}{\begin{proof}[\ensuremath{\mathbf{Proof}}]}
\newcommand{\bs}{\begin{proof}[\ensuremath{\mathbf{Solution}}]}
\newcommand{\ep}{\end{proof}}
\newcommand{\be}{\begin{equation}}
\newcommand{\ee}{\end{equation}}

\begin{document}

\title{On the symmetry and monotonicity of Morrey extremals}

\author{Ryan Hynd\footnote{Department of Mathematics, University of Pennsylvania.  Partially supported by NSF grant DMS-1554130.}\;  and Francis Seuffert\footnote{Department of Mathematics, University of Pennsylvania.}}

\maketitle 

\begin{abstract}
We employ Clarkson's inequality to deduce that each extremal of Morrey's inequality is axially symmetric and is antisymmetric with respect to reflection about a plane orthogonal to its axis of symmetry. We also use symmetrization methods to show that each extremal is monotone in the distance from its axis of symmetry and in the direction of its axis when restricted to spheres centered at the intersection of its axis and its antisymmetry plane. 
\end{abstract}

\section{Introduction} 
Sobolev's inequality asserts for each $p\in (1,n)$, there is a constant $C$ such that 
\be\label{Sobolev}
\left(\int_{\R^n}|u|^{p^*}dx\right)^{1/p^*}\le C\left(\int_{\R^n}|Du|^{p}dx\right)^{1/p}
\ee
for each $u: \R^n\rightarrow\R$ whose first partial derivatives belong to $L^p(\R^n)$ and 
which decays fast enough at infinity. Here 
$$
p^*=\frac{np}{n-p}.
$$
Employing rearrangement methods, Talenti found the smallest 
constant $C^*$  for which \eqref{Sobolev} holds \cite{Talenti}. Talenti also found the {\it Sobolev extremals} or functions for which equality holds in \eqref{Sobolev}  with $C=C^*$; up to scaling, dilating, and translating, they are given by  
$$
u(x)=\frac{1}{(1+|x|^{\frac{p}{p-1}})^{n/p-1}}\quad (x\in \R^n).
$$
In particular, Sobolev extremals are radially symmetric and monotone in the distance from the origin.

\par  Talenti's work on Sobolev's inequality stimulated a lot of interest within the mathematics community.  These results were extended by Aubin for applications in Riemannian geometry \cite{Aubin, Aubin2}. Moreover, these ideas led mathematicians to employ rearrangement methods \cite{Cianchi2, SmWi,Talenti2,VanS}, to seek
best constants \cite{FuMa,Li,SmWi}, and to explore the role of symmetry in various functional inequalities \cite{CaWa, DELT, GaMu, LiWa}. In recent years, researchers have also been using new techniques such as optimal transport to pursue these types of results \cite{Barthe,Cordero, Maggi}. Additionally, a lot of work has been done to quantify these assertions via stability estimates \cite{Bianchi, Carlen, Cianchi, Ne,Seu1,Seu2}.

\par However, much less is known about the equality case of the corresponding inequality for $p\in (n,\infty)$. In this setting, Morrey showed there is $C$ such that 
\be\label{Morrey}
\sup_{x\neq y}\left\{\frac{|u(x)-u(y)|}{|x-y|^{1-n/p}}\right\}\le C\left(\int_{\R^n}|Du|^{p}dx\right)^{1/p}
\ee
for $u\in {\cal D}^{1,p}(\R^n)$ \cite{Morrey}; here $\mathcal{D}^{1,p} (\R^n)$ is the space of weakly differentiable $u: \R^n\rightarrow \R$ for which
$$
u_{x_1},\dots,u_{x_n} \in L^p (\R^n).
$$
This is known as Morrey's inequality. We note that since $p>n$, each $u\in {\cal D}^{1,p}(\R^n)$ agrees almost everywhere with a function that is H\"older continuous with exponent $1-n/p$. Without loss of generality, we may then consider ${\cal D}^{1,p}(\R^n)$ as a subset of the continuous functions on $\R^n$ and identify each $u\in {\cal D}^{1,p}(\R^n)$ with its H\"older continuous representative. 

\par  In recent work, we showed that there is a smallest $C=C_*>0$ so that 
Morrey's inequality holds with $C_*$ and that there exist nonconstant functions for which equality is attained in \eqref{Morrey} with $C$ equal to $C_*$ \cite{HS}.   We will call these functions {\it Morrey extremals}. We also verified that after appropriately rotating, scaling, dilating, and translating a Morrey extremal $u$, it satisfies
\be\label{HolderTwoPoint}
\sup_{x\neq y}\left\{\frac{|u(x)-u(y)|}{|x-y|^{1-n/p}}\right\}=\frac{|u(e_n)-u(-e_n)|}{|e_n-(-e_n)|^{1-n/p}}
\ee
with 
\be\label{upmone}
u(  e_n)=1\;\;\text{and}\;\;u( - e_n)=-1.
\ee
Here $e_n\in \R^n$ is the point with $1$ in its $n$th coordinate and $0$ otherwise. Furthermore, the PDE
\be\label{MorreyPDE}
-\Delta_pu=c(\delta_{e_n}-\delta_{-e_n})
\ee
holds weakly in $\R^n$ for a constant $c>0$. 
\par While $C_*$ and the corresponding Morrey extremals are not explicitly known, many qualitative properties of these functions have been identified. In particular, Morrey extremals which satisfy \eqref{HolderTwoPoint} and \eqref{upmone} are known to be unique, axially symmetric about the $x_n-$axis and antisymmetric about the $x_n=0$ plane (Sections 3 and 6 of \cite{HS}). We established this in our previous work by relying on a uniqueness property of solutions of \eqref{MorreyPDE}. In this paper, we will verify the following theorem as a consequence of Clarkson's inequality.
\begin{thm}\label{mainThm}
Suppose $p>n$, $n\ge 2$, and $u\in {\cal D}^{1,p}(\R^n)$ is the Morrey extremal which satisfies \eqref{HolderTwoPoint} and \eqref{upmone}. Then 
\be\label{AxialSymm}
u(Ox)=u(x),\quad x\in \R^n
\ee 
for each orthogonal transformation $O:\R^n\rightarrow \R^n$ with $Oe_n=e_n$. Moreover,  
\be\label{ReflectSymm}
u(Tx)=-u(x),\quad x\in \R^n
\ee
where $Tx=x-2x_ne_n$.
\end{thm}

\par In addition to the symmetries listed above, the Morrey extremal $u$ which satisfies \eqref{HolderTwoPoint} and \eqref{upmone} has some interesting monotonicity features. The first feature is that $u$ is either nonincreasing or nondecreasing in the distance from the $x_n-$axis.  The key here is that in addition to being axially symmetric $u$ is positive and quasiconcave when restricted to the half space $x_n>0$.
  
 \begin{thm}\label{SchwarzThm}
Assume $p>n$, $n\ge 2$, and  $u\in  {\cal D}^{1,p}(\R^n)$ is the Morrey extremal which satisfies \eqref{HolderTwoPoint} and \eqref{upmone}.
If  $x^1,x^2\in\R^n$ with 
$$
x^1_n=x^2_n\ge 0\;\;\text{and}\;\; |x^2|\le |x^1|,
$$
or if
$$
x^1_n=x^2_n\le 0\;\;\text{and}\;\;|x^1|\le |x^2|,
$$
then
$$
u(x^1)\le u(x^2).
$$
\end{thm}

\par The second monotonicity feature is that $u$ is nondecreasing in the $x_n$ variable when restricted to each sphere centered at the origin. We will use symmetrization methods to prove this and employ a certain P\'olya-Szeg\"o inequality. In particular, we will verify more generally that $u^+$ and $u^-$ are equal to their cap rearrangements as defined in Section \ref{CapSec}. 

 \begin{thm}\label{CapThm}
Assume $p>n$, $n\ge 2$, and  $u\in  {\cal D}^{1,p}(\R^n)$ is the Morrey extremal which satisfies \eqref{HolderTwoPoint} and \eqref{upmone}.
If  $x^1,x^2\in \R^n$ with 
$$
 |x^1|= |x^2|\;\; \text{and}\;\;x^1_n\le x^2_n, 
 $$
then
$$
u(x^1)\le u(x^2).
$$
\end{thm}

\par In what follows, we will prove Theorems \ref{mainThm}, \ref{SchwarzThm}, \ref{CapThm} in sections \ref{basicproofSec}, \ref{SchwarzSec} and \ref{CapSec}, respectively.  In addition, we will take a detour to verify the axial symmetry of Morrey extremals using the ``axial average" and ``axial sweep" transformations presented in section \ref{AltproofSec}. In the appendix, we'll also prove a useful approximation result for functions in ${\cal D}^{1,p}(\R^n)$ with $p>n$. Finally, we would like to thank Eric Carlen, Elliott Lieb, and Peter McGrath for their advice and insightful discussions related to this work.

\begin{figure}
\begin{tikzpicture}

\foreach \Point in {(-2,2), (3,2)}{
    \node[gray] at \Point {\textbullet};}
    
\node at (-1.9,2.4) {$x^2$};

\node at (3.1,2.4) {$x^1$};

\foreach \Point in {(-3,-2.5), (1,-2.5)}{
    \node[gray] at \Point {\textbullet};}
    
\node at (-3,-2.8) {$x^2$};

\node at (1,-2.8) {$x^1$};

\draw[-,thick, gray] (-4,2) to (4,2);

\draw[-,thick, gray] (-4,-2.5) to (4,-2.5);

\draw[->,thick] (0,-4) to (0,4);

\draw[->,thick] (-4,0) to (4,0);

\node at (0.5,4){$x_2$};

\node at (4,0.5){$x_1$};


\hspace{9cm}

\draw[thick,gray] (1,0) arc (0: 360: 1);

\foreach \Point in {(-0.25881, .96592), (0.86602,-.5)}{
    \node[gray] at \Point {\textbullet};}

 \node at (-0.25881, 1.4) {$x^2$};

\node at (1.2,-.5) {$x^1$};

\draw[thick,gray] (3,0) arc (0: 360: 3);

\foreach \Point in {(2.12132034356, 2.12132034356), (0.52094,-2.95442)}{
    \node[gray] at \Point {\textbullet};}

\node at (2.6,2.3) {$x^2$};

\node at (.7,-3.2) {$x^1$};

\draw[->,thick] (0,-4) to (0,4);

\draw[->,thick] (-4,0) to (4,0);

\node at (0.5,4){$x_2$};

\node at (4,0.5){$x_1$};

\end{tikzpicture}
\caption{These diagrams illustrate the monotonicity properties of the Morrey extremal satisfying \eqref{HolderTwoPoint} and \eqref{upmone} for $n=2$.  Theorems \ref{SchwarzThm} and \ref{CapThm} respectively assert that $u(x^1)\le u(x^2)$ 
for $x^1,x^2\in \R^2$ which are ordered as on the horizontal lines in the figure on the left and as on each circle in the figure on the right. }
\end{figure}
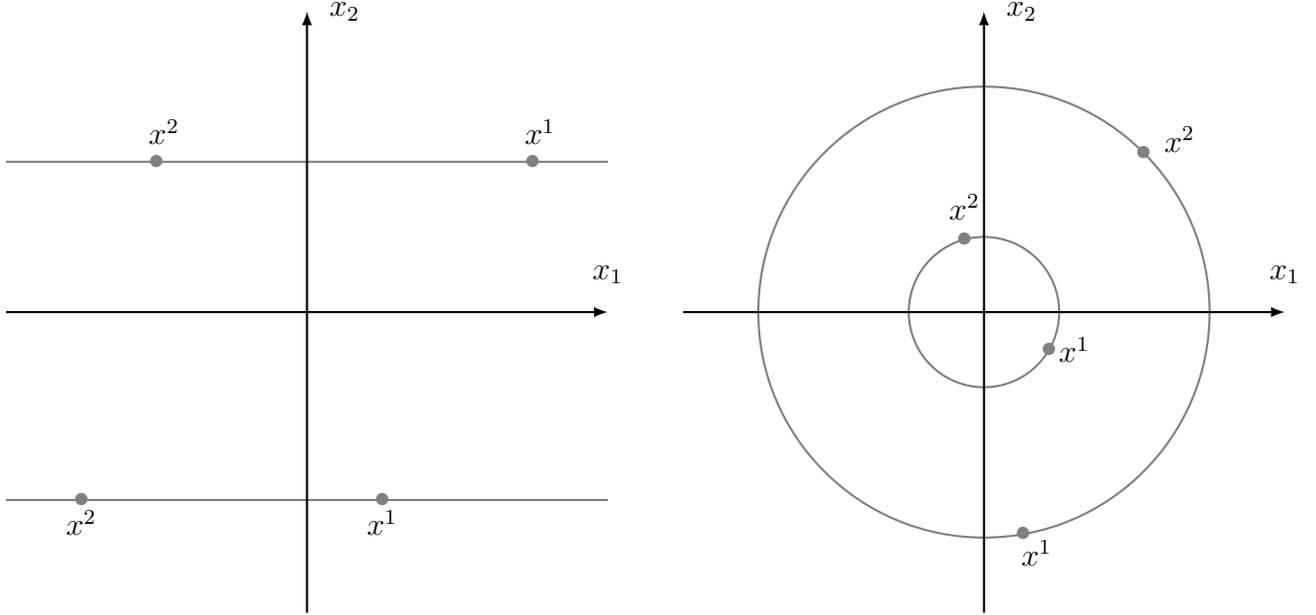

\section{Axial symmetry and reflectional antisymmetry}\label{basicproofSec}
For the remainder of this note, we will suppose 
$$
p>n\quad \text{and}\quad n\ge 2.
$$
We will also use the notation 
\be
[v]_{1-n/p}:=\sup_{x\ne y}\left\{\frac{|v(x)-v(y)|}{|x-y|^{1-n/p}}\right\}
\ee
for the $1-n/p$ H\"older seminorm of $v$. This will allow us to write the sharp form of Morrey's inequality \eqref{Morrey} a bit more concisely as
$$
[v]_{1-n/p}\le C_*\left(\int_{\R^n}|Dv|^pdx\right)^{1/p}.
$$

\par Let us now recall the elementary inequality: for $a,b\in \R^n$, 
\be\label{babyClarkson}
\left|\frac{a+b}{2}\right|^p+
\left |\frac{a-b}{2}\right|^p\le \frac{1}{2}|a|^p+  \frac{1}{2}|b|^p.
\ee
This type of inequality was studied by Clarkson \cite{Clarkson} in connection with uniformly convex spaces, and it immediately implies 
\be\label{Clarkson}
\int_{\R^n}\left
 |\frac{Dv+Dw}{2}\right|^pdx+\int_{\R^n}\left
 |\frac{Dv-Dw}{2}\right|^pdx\le \frac{1}{2}\int_{\R^n}|Dv|^pdx+  \frac{1}{2}\int_{\R^n}|Dw|^pdx
\ee
for $v,w\in {\cal D}^{1,p}(\R^n)$.  A direct consequence is as follows.

\begin{lem}\label{FunLemma}
Suppose $u$ is a Morrey extremal which satisfies \eqref{HolderTwoPoint} and \eqref{upmone}.
Further assume $v\in {\cal D}^{1,p}(\R^n)$ satisfies 
\be
v(  e_n)=1\;\;\text{and}\;\;v( - e_n)=-1.
\ee
and 
$$
\int_{\R^n}|Dv|^pdx\le \int_{\R^n}|Du|^pdx.
$$
Then $u\equiv v$. 
\end{lem}
\begin{proof}
Define 
$$
w:=\frac{u+v}{2}.
$$
Our first assumption on $v$ gives
$$
w(e_n)-w(-e_n)=\frac{v(e_n)-v(-e_n)}{2}+\frac{u(e_n)-u(-e_n)}{2}= u(e_n)-u(-e_n).
$$
It follows that 
$$
[w]_{1-n/p}\ge\frac{|u(e_n)-u(-e_n)|}{|e_n-(-e_n)|^{1-n/p}}=[u]_{1-n/p}.
$$
\par Inequality \eqref{Clarkson} and our second assumption on $v$ imply
\begin{align}
\int_{\R^n}|Dw|^pdx+\int_{\R^n}\left
 |\frac{Du-Dv}{2}\right|^pdx&\le \frac{1}{2}\int_{\R^n}|Du|^pdx+  \frac{1}{2}\int_{\R^n}|Dv|^pdx\\
 &\le \int_{\R^n}|Du|^pdx.
\end{align}
In particular, if $Du\not\equiv Dv$ then 
$$
[u]_{1-n/p}\le [w]_{1-n/p}\le C_*\left(\int_{\R^n}|Dw|^pdx\right)^{1/p}<C_*\left(\int_{\R^n}|Du|^pdx\right)^{1/p}.
$$
However, this would contradict our hypothesis that $u$ is an extremal.   Consequently, there is a constant $c$ such that $v(x)=u(x)+c$ for all $x\in \R^n$. Choosing $x=e_n$ gives 
$$
u(e_n)=  v(e_n)=u(e_n)+c.
$$
That is, $c=0$. 
\end{proof}

\begin{proof}[Proof of Theorem \ref{mainThm}]
Let $O: \R^n\rightarrow \R^n$ be an orthogonal transformation which satisfies $Oe_n=e_n$ and set 
$$
v(x)=u(Ox), \quad x\in \R^n.
$$
Then $v(e_n)=u(e_n)=1$ and $v(-e_n)=u(O(-e_n))=u(-e_n)=-1$. Moreover, 
$$
\int_{\R^n}|Dv|^pdx= \int_{\R^n}|Du|^pdx.
$$
by the change of variables theorem (Theorem 2.44 in \cite{Folland}). In view of Lemma \ref{FunLemma}, $v(x)=u(x)$ for all $x\in \R^n$.  

\par Now set 
$$
w(x)=-u\left(Tx \right)
$$
for $x\in \R^n$, where $Tx= x -2x_ne_n$. As $T(e_n)=- e_n$, 
\be
w(  e_n)=1\;\;\text{and}\;\;w( - e_n)=-1.
\ee
Furthermore, since $T$ is an orthogonal transformation of $\R^n$, we can apply the change of variables theorem again to conclude
$$
\int_{\R^n}|Dw|^pdx= \int_{\R^n}|Du|^pdx.
$$
Lemma \ref{FunLemma} then implies $w(x)=u(x)$ for all $x\in \R^n$. 
\end{proof}

\section{Alternative proofs of axial symmetry}\label{AltproofSec}
In this section, we will use two transformations of ${\cal D}^{1,p}(\R^n)$ functions which result in functions which are axially symmetry with respect to the $x_n-$axis.  To this end, it will be convenient for us to use the variables 
$$
x=(y,x_n)\in \R^{n-1}\times \R
$$
and consider each $u\in{\cal D}^{1,p}(\R^n)$ as the function
$$
(y,x_n)\mapsto u(y,x_n).
$$

\par For $n\ge 3$ and $u\in{\cal D}^{1,p}(\R^n)$, we will also set 
$$
D_ru:=\frac{D_yu\cdot y}{|y|^2}y \quad \text{and}\quad D_{\mathbb{S}^{n-2}}u:=D_yu-D_ru.
$$
This allows us to express the gradient of $u$ as 
$$
Du=(D_ru +D_{\mathbb{S}^{n-2}}u, \partial_{x_n}u).
$$
When $n=2$, we will write
$$
D_{\mathbb{S}^{0}}u(y,x_2):=\frac{\partial_yu(y,x_2)+\partial_yu(-y,x_2)}{2}
$$ 
to stay consistent with our considerations below for $n\ge 3$.

\par In what follows, we'll also make use of the fact that each $u\in \mathcal{D}^{1,p} (\R^n)$ can be approximated by smooth functions. That is, 
for each $u\in  \mathcal{D}^{1,p} (\R^n)$, there is $(u_k)_{k\in \N}\subset C^\infty (\R^n) \cap \mathcal{D}^{1,p} (\R^n)$ 
such that 
$$
\begin{cases}
u_k\rightarrow u\;\text{uniformly on}\; \R^n\\
Du_k\rightarrow Du\; \text{in}\;L^p(\R^n;\R^n)
\end{cases}
$$
as $k\rightarrow\infty$. This assertion likely follows from a general approximation theorem. Nevertheless, we have written a short proof of this fact in Proposition \ref{DensityThm} of the appendix. 

\subsection{Axial average}
For a given $u\in  {\cal D}^{1,p}(\R^n)$, set
$$
u^*(y,x_n):=
\begin{cases}
\displaystyle\fint_{|z|=|y|}u(z,x_n)d\sigma(z), \quad &|y|>0\\
u(0,x_n), \quad & |y|=0
\end{cases}
$$
as its {\it axial average}. Here $\sigma$ is $n-2$ dimensional Hausdorff measure. It is immediate from this definition that $u^*$ is axially symmetric with respect to the $x_n-$axis.  We'll establish a Hardy type inequality and then
 a P\'olya-Szeg\"o type inequality involving $u$ and $u^*$.

\begin{lem}
There is a constant $C$ such that 
\be\label{POineq}
\int_{\R^n}\frac{|u-u^*|^p}{|x-x_ne_n|^p}dx\le C\int_{\R^n}|D_{\mathbb{S}^{n-2}}u|^pdx
\ee
for each $u\in  {\cal D}^{1,p}(\R^n)$. 
\end{lem}
\begin{proof}
First assume $n\ge 3$ and $u\in  C^\infty(\R^n)\cap {\cal D}^{1,p}(\R^n)$.  Recall Poincar\'e's inequality on $\mathbb{S}^{n-2}$: there is a 
constant $C$ such that 
\be\label{Poincare}
\int_{\mathbb{S}^{n-2}}\left|v-\fint_{\mathbb{S}^{n-2}}vd\sigma\right|^pd\sigma\le C\int_{\mathbb{S}^{n-2}}|D_{\mathbb{S}^{n-2}}v|^pd\sigma
\ee
for each $v\in C^\infty(\mathbb{S}^{n-2})$. This inequality can be proved by a minor variation of Theorem 2.10 in \cite{H}. Substituting  $v(\xi)=u(r\xi,x_n)$ gives  
$$
\int_{|y|=r}\left|u(y,x_n)-\fint_{|z|=r}u(z,x_n)d\sigma(z)\right|^pd\sigma(y)\le Cr^p\int_{|y|=r}|D_{\mathbb{S}^{n-2}}u(y,x_n)|^pd\sigma(y)
$$
for each $r>0$.  That is, 
$$
\int_{|y|=r}\frac{|u(y,x_n)-u^*(y,x_n)|^p}{r^p}d\sigma(y)\le C\int_{|y|=r}|D_{\mathbb{S}^{n-2}}u(y,x_n)|^pd\sigma(y).
$$

\par Integrating this inequality over $(r,x_n)\in [0,R]\times[-L,L]$ leads to 
\begin{align}\label{RLineq}
\int_{B_R\times[-L,L]}\frac{|u-u^*|^p}{|x-x_ne_n|^p}dx\nonumber
&= \int^L_{-L}\int^R_0\int_{|y|=r}\frac{\left|u(y,x_n)-u^*(y,x_n)\right|^p}{|y|^p}d\sigma(y)drdx_n\nonumber\\
&\le C\int^L_{-L}\int^R_0\int_{|y|=r}|D_{\mathbb{S}^{n-2}}u(y,x_n)|^pd\sigma(y)drdx_n\nonumber\\
&= C\int_{B_R\times[-L,L]}|D_{\mathbb{S}^{n-2}}u|^pdx\nonumber\\
&\le C\int_{\R^n}|D_{\mathbb{S}^{n-2}}u|^pdx.
\end{align}
Here $B_R:=B_R(0)\subset \R^{n-1}$.  Using Proposition \ref{DensityThm}, it is routine to show \eqref{RLineq} holds for each 
$u\in {\cal D}^{1,p}(\R^n)$. We then conclude \eqref{POineq} by sending $L,R\rightarrow\infty$. 

\par Now suppose $n=2$. As $D_{\mathbb{S}^0}u\in L^p(\R^2)$,
$$
\int_{\R^2}|D_{\mathbb{S}^0}u|^pdx=\int_{\R}\left(\int_{\R}|D_{\mathbb{S}^0}u(y,x_2)|^pdy\right)dx_2<\infty.
$$
Consequently, 
$$
\int_{\R}|D_{\mathbb{S}^0}u(y,x_2)|^pdy<\infty
$$
for almost every $x_2\in \R$. For any such $x_2$, we can apply Hardy's inequality 
\be
\int_{\R}\frac{|f(y)|^p}{|y|^p}dy\le c_p \int_{\R}|f'(y)|^pdy
\ee
to find 
\begin{align}
\int_{\R}\frac{|u(y,x_2)-u^*(y,x_2)|^p}{|y|^p}dy=\int_{\R}\frac{|(u(y,x_2)-u(-y,x_2))/2|^p}{|y|^p}dy\le c_p \int_{\R}|D_{\mathbb{S}^0}u(y,x_2)|^pdy.
\end{align}
The inequality \eqref{POineq} now follows from integrating over $x_2$. 
\end{proof}

\begin{prop}
For all $u\in  {\cal D}^{1,p}(\R^n)$, 
\be\label{PolyaSzego}
\int_{\R^n}|Du^*|^pdx\le \int_{\R^n}|Du|^pdx.
\ee
Equality holds if and only if $u=u^*$. 
\end{prop}
\begin{proof}
Let us first assume $n\ge 3$ and that $u\in  C^\infty(\R^n)\cap {\cal D}^{1,p}(\R^n)$. Direct computation gives 
\be
Du^*(y,x_n)=\fint_{|z|=r}\left(\left[D_yu(z,x_n)\cdot \frac{z}{r}\right] \frac{y}{r},\partial_{x_n}u(z,x_n)\right)d\sigma(z)
\ee
for $r=|y|>0$; and by Jensen's inequality, 
\begin{align}
|Du^*(y,x_n)|^p&\le \fint_{|z|=r}\left|\left(\left[D_yu(z,x_n)\cdot \frac{z}{r}\right] \frac{y}{r},\partial_{x_n}u(z,x_n)\right)\right|^pd\sigma(z)\\
&= \fint_{|z|=r}|(D_ru(z,x_n),\partial_{x_n}u(z,x_n)|^pd\sigma(z).
\end{align}
It follows that 
\begin{align}
\int_{\R^n}|Du^*|^pdx&=\int_{\R}\int_{\R^{n-1}}|Du^*(y,x_n)|^pdydx_n\\
&=\int_{\R}\int^\infty_0\left(\int_{|y|=r}|Du^*(y,x_n)|^pd\sigma(y)\right)drdx_n\\
& \le \int_{\R}\int^\infty_0\int_{|y|=r}\left(\fint_{|z|=r}|(D_ru(z,x_n),\partial_{x_n}u(z,x_n))|^pd\sigma(z)\right)d\sigma(y)drdx_n\\
&= \int_{\R}\int^\infty_0\int_{|z|=r}|(D_ru(z,x_n),\partial_{x_n}u(z,x_n))|^pd\sigma(z)drdx_n\\
& = \int_{\R}\int_{\R^{n-1}}|(D_ru(z,x_n),\partial_{x_n}u(z,x_n))|^pdzdx_n\\
& = \int_{\R^n}|(D_ru,\partial_{x_n}u)|^pdx.
\end{align}

\par Combining this inequality with Proposition \ref{DensityThm}, we deduce
\be\label{ustarkayest}
\int_{\R^n}|Du^*|^pdx\le \int_{\R^n}|(D_ru,\partial_{x_n}u)|^pdx
\ee
for each $u\in {\cal D}^{1,p}(\R^n)$. Employing the elementary inequality 
$$
|a|^p+|b|^p\le (|a|^2+|b|^2)^{p/2},\quad a,b\in \R^n
$$
with $a=(D_ru,u_{x_n})$ and $b=(D_{\mathbb{S}^{n-2}}u,0)$ gives
\be\label{ElemGradEst}
|(D_ru,u_{x_n})|^p+|D_{\mathbb{S}^{n-2}}u|^p\le |Du|^p.
\ee
It follows that
\be\label{ustarinfinty}
\int_{\R^n}|Du^*|^pdx\le \int_{\R^n}|Du|^pdx-\int_{\R^n}|D_{\mathbb{S}^{n-2}}u|^pdx.
\ee
Consequently, if equality holds in \eqref{PolyaSzego},  
$$
\int_{\R^n}|D_{\mathbb{S}^{n-2}}u|^pdx=0.
$$
We can then appeal to \eqref{POineq} to find $u=u^*$.

\par Now suppose $n=2$. Here
\be\label{ustarnequal2}
u^*(y,x_2)=\frac{u(y,x_2)+u(-y,x_2)}{2}
\ee
and 
\be
Du^*(y,x_2)=\frac{Du(y,x_2)+(-\partial_yu(-y,x_2),\partial_{x_2}u(-y,x_2))}{2}.
\ee
By Clarkson's inequality \eqref{Clarkson},  
\begin{align*}
\int_{\R^2}|Du^*|^pdx&=\iint_{\R^2}\left|\frac{Du(y,x_2)+(-\partial_yu(-y,x_2),\partial_{x_2}u(-y,x_2))}{2}\right|^pdydx_2\\
&\le\frac{1}{2}\iint_{\R^2}|Du(y,x_2)|^pdydx_2+\frac{1}{2}\iint_{\R^2}|(-\partial_yu(-y,x_2),\partial_{x_2}u(-y,x_2))|^pdydx_2\\ 
&\hspace{.5in} -\iint_{\R^2}\left|\frac{Du(y,x_2)-(-\partial_yu(-y,x_2),\partial_{x_2}u(-y,x_2))}{2}\right|^pdydx_2\\
&=\frac{1}{2}\iint_{\R^2}|Du(y,x_2)|^pdydx_2+\frac{1}{2}\iint_{\R^2}|Du(-y,x_2)|^pdydx_2 \\
&\hspace{.5in} -\iint_{\R^2}\left|\left(\frac{\partial_yu(y,x_2)+\partial_yu(-y,x_2)}{2},\frac{\partial_{x_2}u(y,x_2)- \partial_{x_2}u(-y,x_2)}{2}\right)\right|^pdydx_2\\
&\le\int_{\R^2}|Du|^pdx-\int_{\R^2}|D_{\mathbb{S}^{0}}u|^pdx.
\end{align*}
If equality holds in \eqref{PolyaSzego}, we can again appeal to \eqref{POineq} to find $u=u^*$.
\end{proof}
As $u=u^*$ along the $x_n-$axis, the following corollary follows directly from Lemma \ref{FunLemma} and inequality (\ref{PolyaSzego}). 
\begin{cor}
Suppose $u\in  {\cal D}^{1,p}(\R^n)$ is a Morrey extremal which satisfies \eqref{HolderTwoPoint} and \eqref{upmone}. Then $u=u^*$.  
\end{cor}

\subsection{Axial sweep}
For a given $u \in{\cal D}^{1,p}(\R^n)$, we will also consider its {\it axial sweep}
\be
u^\zeta(y,x_n):=u(|y|\zeta,x_n)
\ee
 with respect to a direction $\zeta\in\mathbb{S}^{n-2}$.   Clearly, $u^\zeta$ is axially symmetric for each $\zeta$. In this section, we prove the following assertion. 
\begin{prop}\label{AxialSweepCor}
Suppose $u \in{\cal D}^{1,p}(\R^n)$. There is 
$\zeta\in  \mathbb{S}^{n-2}$ for which 
\be\label{AxialSweepIneq2}
\int_{\R^n}|Du^\zeta|^pdx\le \int_{\R^n}|Du|^pdx.
\ee
If $n\ge 3$ and $u$ is not axially symmetric, $\zeta$ can be chosen so that this inequality is strict. 
\end{prop}
As $u=u^\zeta$ along the $x_n-$axis, it follows immediately from Lemma \ref{FunLemma} and Proposition \ref{AxialSweepCor} that each extremal is axially symmetric.
\begin{cor}
Suppose $u\in  {\cal D}^{1,p}(\R^n)$ is a Morrey extremal which satisfies \eqref{HolderTwoPoint} and \eqref{upmone}. Then $u=u^\zeta$ for every $\zeta\in \mathbb{S}^{n-2}$.  
\end{cor}

The key to proving Proposition \ref{AxialSweepCor} is the following inequality.
\begin{lem}
Suppose $u \in{\cal D}^{1,p}(\R^n)$. Then $u^\zeta\in{\cal D}^{1,p}(\R^n)$ for $\sigma$ almost every $\zeta\in\mathbb{S}^{n-2}$ 
and
\be\label{AxialSweepIneq}
 \displaystyle\fint_{|\zeta|=1}\left(\int_{\R^n}|Du^\zeta|^pdx\right)d\sigma(\zeta)\le \int_{\R^n}|Du|^pdx.
\ee
If equality holds and $n\ge 3$, $u=u^\zeta$ for each $\zeta\in\mathbb{S}^{n-2}$.
\end{lem}
\begin{proof}
1. First suppose $n=2$. Note  $u^{\pm 1}(y,x_2)=u(\pm |y|,x_2)$.  Direct computation shows
$$
\int_{\R^2}|Du^1|^pdx=2\int_{\R}\int^\infty_0|Du(y,x_2)|^pdydx_2
$$
and 
$$
\int_{\R^2}|Du^{-1}|^pdx=2\int_{\R}\int^{0}_{-\infty}|Du(y,x_2)|^pdydx_2.
$$
Therefore, 
\be
 \displaystyle\fint_{|\zeta|=1}\left(\int_{\R^n}|Du^\zeta|^pdx\right)d\sigma(\zeta)=\int_{\R^2}\frac{|Du^{1}|^p+|Du^{-1}|^p}{2}dx=\int_{\R^2}|Du|^pdx.
\ee

\par 2. Now let's assume $n\ge3 $ and that $u \in C^\infty(\R^n)\cap {\cal D}^{1,p}(\R^n)$. Note that for $r=|y|>0$, 
\be
Du^\zeta(y,x_n)=\left(\left[D_yu(r\zeta,x_n)\cdot \zeta\right] \frac{y}{r},\partial_{x_n}u(r\zeta,x_n)\right)
\ee
and
\be
|Du^\zeta(y,x_n)|=|(D_ru(r\zeta,x_n),\partial_{x_n}u(r\zeta,x_n))|.
\ee
As a result, 
\begin{align}
\fint_{|\zeta|=1}|Du^\zeta(y,x_n)|^pd\sigma(\zeta)&= \fint_{|\zeta|=1}\left|(D_ru(r\zeta,x_n),\partial_{x_n}u(r\zeta,x_n))\right|^pd\sigma(\zeta)\\
&= \fint_{|z|=r}\left|(D_ru(z,x_n),\partial_{x_n}u(z,x_n))\right|^pd\sigma(z).
\end{align}

\par Observe that $|Du^\zeta(y,x_n)|$ is continuous on set of $(y,x_n,\zeta)\in(\R^{n-1}\setminus\{0\})\times \R\times \mathbb{S}^{n-2}$. So we can apply Fubini's theorem and conclude
\begin{align}
&\fint_{|\zeta|=1}\left(\int_{\R^{n}}|Du^\zeta|^pdx\right)d\sigma(\zeta)\\
&\hspace{1in} = \int_{\R^n}\left(\fint_{|\zeta|=1}|Du^\zeta|^pd\sigma(\zeta)\right)dx\\
&\hspace{1in} = \int_{\R}\int^\infty_0\int_{|y|=r}\left(\fint_{|\zeta|=1}|Du^\zeta(y,x_n)|^pd\sigma(\zeta)\right)d\sigma(y)drdx_n\\
&\hspace{1in} = \int_{\R}\int^\infty_0\int_{|y|=r}\left(\fint_{|z|=r}\left|(D_ru(z,x_n),\partial_{x_n}u(z,x_n))\right|^pd\sigma(z)\right)d\sigma(y)drdx_n\\
&\hspace{1in} = \int_{\R}\int^\infty_0\int_{|z|=r}|(D_ru(z,x_n),\partial_{x_n}u(z,x_n))|^pd\sigma(z)drdx_n\\
&\hspace{1in} = \int_{\R}\int_{\R^{n-1}}|(D_ru(z,x_n),\partial_{x_n}u(z,x_n))|^pdzdx_n\\
&\hspace{1in} = \int_{\R^n}|(D_ru,\partial_{x_n}u)|^pdx.
\end{align}

\par 3. For $u \in{\cal D}^{1,p}(\R^n)$, we can select $(u_k)_{k\in \N}\subset  C^\infty(\R^n)\cap {\cal D}^{1,p}(\R^n)$ such that 
$u_k\rightarrow u$ uniformly in $\R^n$ and $Du_k\rightarrow Du$ in $L^p(\R^n; \R^n)$ as $k\rightarrow\infty$. It is also easy to check from the definition that $u^\zeta_k\rightarrow u^\zeta$ uniformly on $\R^n$ for each $\zeta\in\mathbb{S}^{n-2}$. We define 
$$
v_k(y,x_n,\zeta)=
\begin{cases}
Du_k^\zeta(y,x_n), \quad &y\neq 0\\
0,\quad &y=0
\end{cases}
$$
for $(y,x_n,\zeta)\in\R^{n-1}\times \R\times \mathbb{S}^{n-2}$.
By the estimate derived above 
\begin{align}
\iint_{\R^{n}\times \mathbb{S}^{n-2}}|v_k|^pd({\cal L}\times \sigma)=\int_{|\zeta|=1}\left(\int_{\R^{n}}|Du_k^\zeta|^pdx\right)d\sigma(\zeta)= \sigma( \mathbb{S}^{n-2})
\int_{\R^n}|(D_ru_k,\partial_{x_n}u_k)|^pdx.
\end{align}
Here ${\cal L}$ denotes Lebesgue measure on $\R^n$. 

\par Since $Du_k\rightarrow Du$ in $L^p(\R^n; \R^n)$, 
\be
\sup_{k\in \N}\iint_{\R^{n}\times \mathbb{S}^{n-2}}|v_k|^pd({\cal L}\times \sigma)<\infty.
\ee
As a result, there is a measurable $v: \R^{n}\times \mathbb{S}^{n-2} \rightarrow \R^n$ with $|v|\in L^p( \R^{n}\times \mathbb{S}^{n-2};{\cal L}\times \sigma)$ and a subsequence $(v_{k_j})_{j\in \N}$ such that
\be
\lim_{j\rightarrow\infty}\iint_{\R^{n}\times \mathbb{S}^{n-2}}v_{k_j}\cdot \varphi d({\cal L}\times \sigma)
=\iint_{\R^{n}\times \mathbb{S}^{n-2}}v\cdot \varphi d({\cal L}\times \sigma)
\ee
for each measurable $\varphi: \R^{n}\times \mathbb{S}^{n-2}\rightarrow \R^n$ with $|\varphi|\in L^{p/(p-1)}( \R^{n}\times \mathbb{S}^{n-2};{\cal L}\times \sigma)$.

\par In view of this weak convergence, we also have 
\begin{align}\label{WeakConvLimit}
\iint_{\R^{n}\times \mathbb{S}^{n-2}}|v|^pd({\cal L}\times \sigma)&\le \liminf_{j\rightarrow\infty}\iint_{\R^{n}\times \mathbb{S}^{n-2}}|v_{k_j}|^p d({\cal L}\times \sigma)\\
&\le \sigma( \mathbb{S}^{n-2}) \int_{\R^n}|(D_ru,\partial_{x_n}u)|^pdx.
\end{align}
We can apply Fubini's theorem once again to find
\begin{align}\label{FubiniApp}
\int_{|\zeta|=1}\left(\int_{\R^{n}}|v^\zeta(x)|^pdx\right)d\sigma(\zeta)&\le \sigma( \mathbb{S}^{n-2}) \int_{\R^n}|(D_ru,\partial_{x_n}u)|^pdx,
\end{align}
where $v^\zeta:=v(\cdot,\zeta)\in L^p(\R^n;\R^n)$ for $\sigma$ almost every $\zeta\in \mathbb{S}^{n-2}$ (Theorem 2.37 of \cite{Folland}). 

\par 4. We claim 
\be\label{SweepClaim}
Du^\zeta=v^\zeta\; \text{for $\sigma$ almost every $\zeta$}.
\ee
Inequality \eqref{AxialSweepIneq} would then follow from \eqref{FubiniApp}. To this end, we let $\phi\in C^1_c(\R^n;\R^n)$, $\eta\in C( \mathbb{S}^{n-2})$, and integrate by parts to get 
$$
\int_{\R^n\times \mathbb{S}^{n-2}}v_k\cdot \phi \eta\; d({\cal L}\times\sigma)=\int_{\mathbb{S}^{n-2}}\left(\int_{\R^n}Du^\zeta_k\cdot \phi dx \right)\eta d\sigma(\zeta)=-\int_{\mathbb{S}^{n-2}}\left(\int_{\R^n}u^\zeta_k\; \text{div}(\phi) dx \right)\eta d\sigma(\zeta).
$$
Since $\text{div}(\phi)$ has compact support and $u_k\rightarrow u$ uniformly in $\R^n$, we find in the limit as $k=k_j\rightarrow\infty$ that
$$
\int_{\R^n\times \mathbb{S}^{n-2}}v\cdot \phi \eta d({\cal L}\times\sigma)=-\int_{\mathbb{S}^{n-2}}\left(\int_{\R^n}u^\zeta\; \text{div}(\phi) dx \right)\eta d\sigma(\zeta).
$$
That is, 
$$
\int_{\mathbb{S}^{n-2}}\left(\int_{\R^n}v^\zeta\cdot \phi dx \right)\eta d\sigma(\zeta)=-\int_{\mathbb{S}^{n-2}}\left(\int_{\R^n}u^\zeta\; \text{div}(\phi) dx \right)\eta d\sigma(\zeta).
$$
And as $\eta$ is arbitrary
\be\label{SweepClaimzero}
\int_{\R^n}v^\zeta\cdot \phi dx=-\int_{\R^n}u^\zeta\; \text{div}(\phi) dx
\ee
for $\sigma$ almost every $\zeta\in \mathbb{S}^{n-2}$.  Since $C^1_c(\R^n;\R^n)$ equipped with the norm 
$$
\phi\mapsto\max\{\|\phi\|_{\infty},\|D\phi\|_{\infty}\}
$$
is separable, \eqref{SweepClaimzero} holds for a subset of $\mathbb{S}^{n-2}$ with full measure that is independent of $\phi$. We conclude \eqref{SweepClaim}.

\par 5. Observe that we have established
\be\label{AxialSweepIneqNbigThree}
\displaystyle\fint_{|\zeta|=1}\left(\int_{\R^n}|Du^\zeta|^pdx\right)d\sigma(\zeta)\le \int_{\R^n}|(D_ru,\partial_{x_n}u)|^pdx\le \int_{\R^n}|Du|^pdx-\int_{\R^n}|D_{\mathbb{S}^{n-2}}u|^pdx
\ee
for $n\ge 3$. So if equality holds in \eqref{AxialSweepIneq} and $n\ge 3$, 
$$
\int_{\R^n}|D_{\mathbb{S}^{n-2}}u|^pdx=0.
$$
In view of \eqref{POineq},  $u$ is axially symmetric with respect to the $x_n-$axis and so $u=u^\zeta$ for each $\zeta\in \mathbb{S}^{n-2}$.
\end{proof}

\begin{proof}[Proof of Proposition \ref{AxialSweepCor}]
In view of inequality \eqref{AxialSweepIneq}, there is a subset $S\subset \mathbb{S}^{n-2}$ for which 
$\sigma(S)>0$ and \eqref{AxialSweepIneq2} holds for $\sigma$ almost every $\zeta\in S$.  When $n\ge 3$, we have
the refinement \eqref{AxialSweepIneqNbigThree} which gives a subset $S\subset \mathbb{S}^{n-2}$ of positive measure such that 
$$
\int_{\R^n}|Du^\zeta|^pdx\le \int_{\R^n}|Du|^pdx-\int_{\R^n}|D_{\mathbb{S}^{n-2}}u|^pdx.
$$
for $\sigma$ almost every $\zeta\in S$. If $u$ is not axially symmetric, $\int_{\R^n}|D_{\mathbb{S}^{n-2}}u|^pdx>0$ and \eqref{AxialSweepIneq2} 
is strict. 
\end{proof}

\section{Monotonicity from the axis of symmetry}\label{SchwarzSec}
This section is dedicated to proving Theorem \ref{SchwarzThm}. To this end, let us suppose $u$ is the Morrey extremal which satisfies \eqref{HolderTwoPoint} and \eqref{upmone}. We established in previous work that when $u$ is restricted to the half space $x_n>0$, it assumes values between $(0,1]$ and is quasiconcave (Sections 3 and 4 of \cite{HS}).  As a result, 
$\{u\ge c\}$ is a convex subset of the $x_n>0$ half space for each $c\in (0,1]$.  Furthermore, we established the limit
$$
\lim_{|x|\rightarrow\infty}u(x)=0
$$ 
in \cite{HS2}. Consequently, we have that $\{u\ge c\}$ is also compact for each $c\in (0,1]$. 

\par Since $u$ is axially symmetric, 
$$
O(\{u\ge c\})=\{u\ge c\}
$$
for all 
orthogonal transformations $O: \R^n\rightarrow \R^n$ such that $Oe_n=e_n$.  It then follows that for $c\in (0,1]$ and $a>0$,
$$
\{y\in \R^{n-1}: u(y,a)\ge c\}
$$
is a closed ball in $\R^{n-1}$ centered at the origin whenever it is nonempty. This will be 
the crucial observation we use to prove Theorem \ref{SchwarzThm}.

\begin{proof}[Proof of Theorem \ref{SchwarzThm}]
\par Suppose $y_1,y_2\in \R^{n-1}$ with 
\be\label{yoneytwo}
|y_2|\le |y_1|
\ee
and $a>0$. Set 
$$
c=u(y_1,a).
$$
Note that since $a>0$, $c\in (0,1]$. Therefore, 
$$
y_1\in \{y\in \R^{n-1}: u(y,a)\ge c\}=\overline{B_r(0)}\subset \R^{n-1}
$$
for some $r\ge 0$. By \eqref{yoneytwo}, $y_2\in\overline{B_r(0)}$, as well. Consequently, 
$$
y_2\in \{y\in \R^{n-1}: u(y,a)\ge c\}
$$
and in turn 
$$
u(y_2,a)\ge c= u(y_1,a).
$$

\par The assertion 
$$
u(y_1,a)\le u(y_2,a)
$$
when
$$
|y_1|\le |y_2|\;\;\text{and}\;\; a<0
$$
follows similarly.  Finally, the conclusion of Theorem \ref{mainThm} that $u(x - 2 x_n e_n) = -u(x)$ for $x \in \R^n$  implies $u(y,0)=0$ for all $y\in \R^{n-1}$.
\end{proof}

\section{Cap symmetry}\label{CapSec}
In this section, we will recall the notion of the cap symmetrization of a subset of $\R^n$.  This leads naturally to a way to rearrange the values of a function on $\R^n$.  It turns out that the positive and negative parts of the Morrey extremal we have been studying in this 
paper is invariant under this rearrangement process. A key object in our study will be the {\it spherical cap} 
\be
C_{t,\theta}:=\{x\in \R^n: |x|=t,\; x_n>t\cos\theta\}
\ee
with radius $t\ge 0$ and opening angle $\theta\in [0,\pi]$.  We note that for $\theta>0$, $C_{t,\theta}$ is an open subset of the sphere $\partial B_t$. 

\par The following definition of cap symmetrization is due to Sarvas \cite{Sarvas} (see also Brock and Solynin \cite{BrockSolynin}).  Observe that we will also change notation and now use $\sigma$ to denote $n-1$ dimensional Hausdorff measure. 
\begin{SCfigure}
\caption{The diagram on the left depicts the spherical cap $C_{t,\theta}$ of radius $t$ and opening angle $\theta$. The cap contains all points $x \in \R^2$ such that $|x| = t$ and $x_2 > t \cos \theta$.}
\begin{tikzpicture}

\node at (-2,3){$C_{t,\theta}$};

\draw[very thick] (0,0) -- (2.121,-2.121) arc (-45 : 225: 3cm) -- cycle;

\draw[fill=white] (0,0) -- (2.975,0) arc (0 : 360: 2.975cm) -- cycle;

\draw[dashed] (0,0) -- (2.121,-2.121);

\draw[dashed] (0,0) -- (-2.121,-2.121);

\draw (0,0) -- (2.853,-0.927);

\draw (0,1) arc (90:225: 1cm);

\draw[fill=black] (2.853,-0.927) circle (.11cm) ;

\draw[->,thick] (0,-3.35) to (0,4);

\draw[->,thick] (-4,0) to (4,0);

\node at (-1.25,-0.2){$\theta$};

\node at (3.15,-0.927){$x$};

\node at (0.5,4){$x_2$};


\node at (4,0.5){$x_1$};

\node at (3.25,.25){$t$};

\end{tikzpicture}
\end{SCfigure}
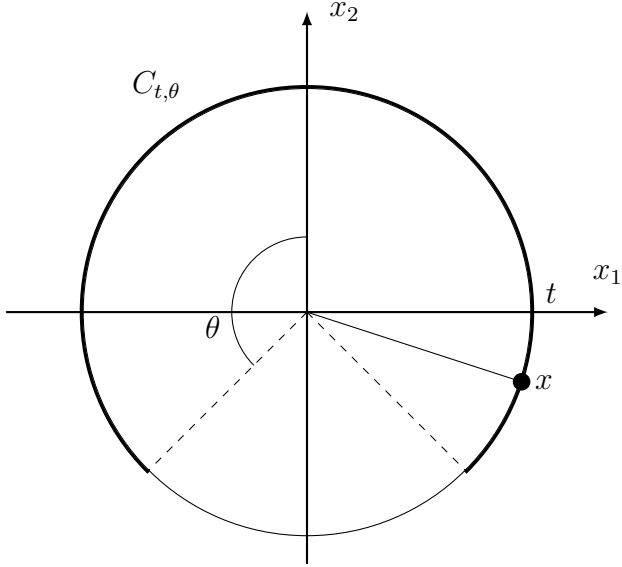

\begin{defn}
Suppose $A\subset \R^n$ is an open subset. The {\it cap symmetrization of $A$ with respect to the 
positive $x_n-$axis} is the subset $A^\star\subset\R^n$ which satisfies 
$$
A^\star\cap \partial B_t=
\begin{cases}
\emptyset  &\text{if}\; A\cap \partial B_t=\emptyset\\
 \partial B_t &\text{if}\; A\cap \partial B_t=\partial B_t\\
 C_{t,\theta}&\text{otherwise}
\end{cases}
$$
for each $t\ge 0$. If $A^\star\cap \partial B_t=C_{t,\theta}$, $\theta\in [0,\pi]$ is chosen so that 
$$
\sigma(A\cap \partial B_t)=\sigma(C_{t,\theta}).
$$
\end{defn}

\par Since $A^\star\cap \partial B_t$ is specified for each $t\ge 0$, this uniquely defines $A^\star$ when $A$ is open. 
If $A\subset \R^n$ is closed, we define $A^\star$ as above with
$$
\overline{C_{t,\theta}}=\{x\in \R^n: |x|=t,\; x_n\ge t\cos\theta\}
$$
replacing $C_{t,\theta}$. It's plain to see that if $A\subset B$, then $A^\star\subset B^\star$. It is also not hard to 
deduce the implication 
\be\label{HalfSpaceProp}
A\subset\{x_n>0\}\Longrightarrow A^\star\subset\{x_n>0\}.
\ee
Moreover, 
it is known that if $A$ is open, $A^\star$ is open; and if $A$ is closed, $A^\star$
is closed (section 2 of \cite{Sarvas}). Furthermore, we can apply the co-area 
formula to conclude $A^\star$ and $A$ have the same Lebesgue measure.

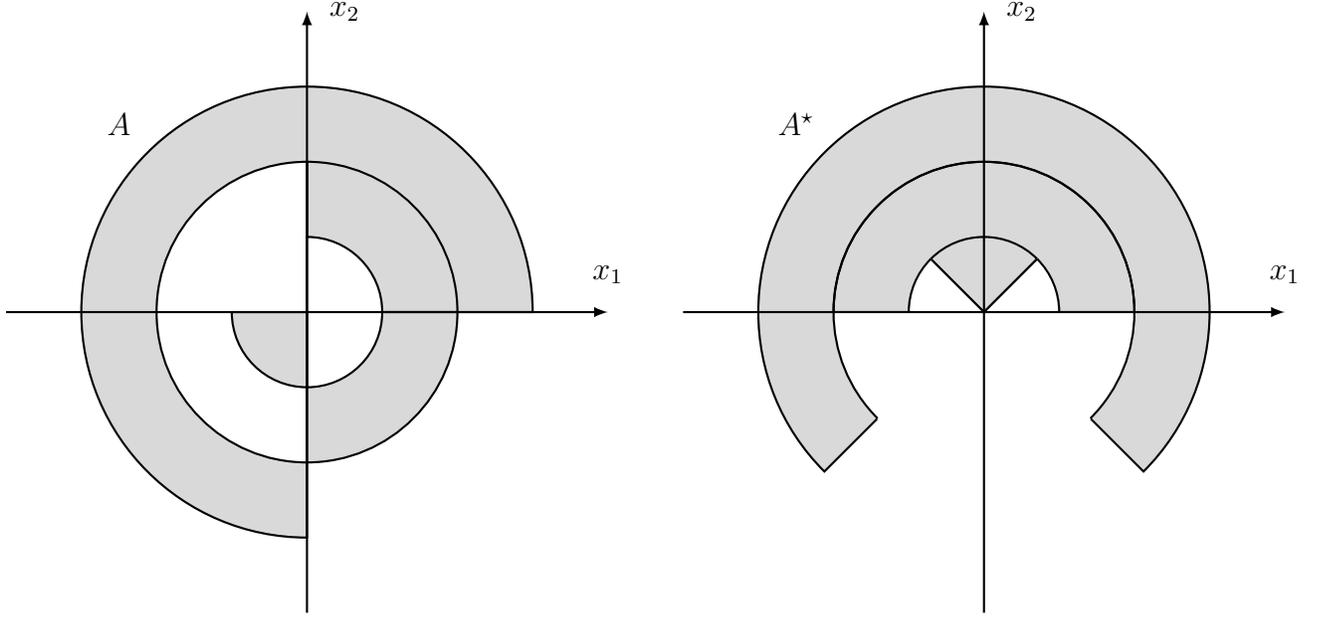
\begin{figure}
\begin{tikzpicture}

\node at (-2.5,2.5) {$A$};

\draw[thick,fill=gray!30] (0,0) -- (3,0) arc  (0 : 270: 3cm) -- cycle;

\draw[thick] (0,0) -- (2,0) arc (0 : 90: 2cm) -- cycle;

\draw[thick,fill=white] (0,0) -- (0,2) arc  (90 : 270: 2cm) -- cycle;

\draw[thick,fill=gray!30] (0,0) -- (0,-2) arc  (-90 : 0: 2cm) -- cycle;

\draw[thick,fill=white] (0,0) -- (0,-1) arc  (-90 : 90: 1cm) -- cycle;

\draw[thick,fill=gray!30] (0,0) -- (-1,0) arc (180: 270: 1cm) --cycle;

\draw[->,thick] (0,-4) to (0,4);

\draw[->,thick] (-4,0) to (4,0);

\node at (0.5,4){$x_2$};

\node at (4,0.5){$x_1$};

\hspace{9cm}

\node at (-2.5,2.5) {$A^\star$};

\draw[thick,fill=gray!30] (0,0) -- (2.121,-2.121) arc  (-45 : 225: 3cm) -- cycle;

\draw[thick,fill=gray!30] (0,0) -- (2,0) arc  (0: 180: 2cm) --cycle;

\draw[thick,fill=gray!30] (0,0) -- (.7071,.7071) arc (45: 135: 1cm) --cycle;

\draw[fill=white,white] (0,0) -- (-2,0) arc (180: 360:2cm) --cycle;

\draw[fill=white,white] (0,0) -- (1.4142,-1.4142) arc (-45: -135:2cm) --cycle;

\draw[fill=white,white] (0,0) -- (1,0) arc (0: 45:1cm) --cycle;

\draw[fill=white,white] (0,0) -- (1,0) arc (0: 45:1cm) --cycle;

\draw[fill=white,white] (0,0) -- (-.70711,.70711) arc (135: 180:1cm) --cycle;

\draw[->,thick] (0,-4) to (0,4);

\draw[thick] (0,0) to (.7071,.7071);

\draw[thick] (0,0) to (-.7071,.7071);

\draw[thick] (1,0) arc (0 : 45: 1cm);

\draw[thick] (-1,0) arc (180 : 135: 1cm);

\draw[thick] (1.4142,-1.4142) arc (-45: 225: 2cm);

\draw[->,thick] (-4,0) to (4,0);

\node at (0.5,4){$x_2$};

\node at (4,0.5){$x_1$};

\end{tikzpicture}
\caption{The shaded region on the left represents a (closed) subset $A \subset \mathbb{R}^2$. The shaded region on the right is $A^\star\subset \mathbb{R}^2$,  the cap symmetrization of $A$ in the direction of the positive $x_2$ axis.}
\end{figure}

\begin{defn}
 We say $v\in {\cal D}^{1,p}(\R^n)$ is {\it admissible} if $\{v>c\}$ has finite Lebesgue measure  
for each $c>\inf v$. In this case, we set 
$$
v^\star(x):=\sup\left\{c>\inf v: x\in \{v>c\}^\star\right\},\quad x\in \R^n
$$
as the {\it cap rearrangement} of $v$ with respect to the $x_n$ axis.  
\end{defn}
It is known that since $v\in {\cal D}^{1,p}(\R^n)$ is continuous, $v^\star$ is continuous with 
$$
\{v^\star>c\}=\{v>c\}^\star
$$
for $c>\inf v$ (section 3 of \cite{BrockSolynin}).  We will show that $v^\star$ is axially symmetric and also that it has the monotonicity property as described 
in Theorem \ref{CapThm}. Then we will explain how this translates to Morrey extremals. 
\begin{lem}
Suppose $v\in {\cal D}^{1,p}(\R^n)$ is admissible. Then $v^\star$ is axially symmetric with respect to the $x_n-$axis. 
\end{lem}
\begin{proof}
Suppose $O: \R^n\rightarrow \R^n$ is an orthogonal transformation with $Oe_n=e_n$. Observe 
$$
O^{-1}(\partial B_t)=\partial B_t\;\text{and}\; O^{-1}(C_{t,\theta})=C_{t,\theta}
$$ 
for each $t\ge 0$ and $\theta\in [0,\pi]$. As a result, 
\begin{align}
\{v^\star\circ O>c\}\cap \partial B_t &=O^{-1}(\{v^\star>c\})\cap \partial B_t\\
&=O^{-1}(\{v^\star>c\})\cap O^{-1}(\partial B_t)\\
&=O^{-1}\left(\{v^\star>c\}\cap \partial B_t\right)\\
&=\{v^\star>c\}\cap \partial B_t
\end{align}
for each $t\ge 0$ and $c>\inf v^\star$.  Consequently, 
$$
\{v^\star\circ O>c\}=\{v^\star>c\}
$$
for all $c>\inf v^\star$ and in turn $v^\star\circ O=v^\star$.
\end{proof}
\begin{lem}\label{capDecreaseLemma}
Suppose $v\in {\cal D}^{1,p}(\R^n)$ is admissible. If $x^1,x^2\in \R^n$ satisfies 
\be\label{x1x2incap}
 |x^1|= |x^2|\;\; \text{and}\;\;x^1_n\le x^2_n,
\ee
then
$$
v^\star(x^1)\le v^\star(x^2).
$$
\end{lem}
\begin{proof}
Set $c=v^\star(x^1)$ and $R=|x^1|$.  Recall that 
$$
\{v^\star\ge c\}\cap \partial B_R=\{v\ge c\}^\star\cap \partial B_R=\overline{C_{R,\theta}}
$$
for some $\theta\in [0,\pi].$ By \eqref{x1x2incap}, $x^2 \in \overline{C_{R,\theta}}$. Consequently, 
$$
v^\star(x^2)\ge c=v^\star(x^1).
$$
\end{proof}
\begin{cor}\label{capMaxLemma}
Suppose $v\in {\cal D}^{1,p}(\R^n)$ is admissible.  Then 
$$
v^\star(e_n)=\max_{\partial B_1}v.
$$
\end{cor}
\begin{proof}
Since $\{v>\max_{\partial B_1}v\}\cap \partial B_1=\emptyset$, 
$$
\{v>\max_{\partial B_1}v\}^\star\cap \partial B_1=\{v^\star>\max_{\partial B_1}v\}\cap \partial B_1=\emptyset.
$$
It follows that  
$$
v^\star(e_n)\le \max_{\partial B_1}v.
$$

\par For any $\epsilon>0$, $\{v>\max_{\partial B_1}v-\epsilon\}\cap\partial B_1$ is nonempty and open. Therefore, 
$$
\{v>\max_{\partial B_1}v-\epsilon\}^\star\cap \partial B_1=\{v^\star>\max_{\partial B_1}v-\epsilon\}\cap \partial B_1\supset C_{1,\theta}
$$ 
for some $\theta\in (0,\pi]$. It follows that $e_n\in \{v^\star>\max_{\partial B_1}v-\epsilon\}\cap \partial B_1$ and so
$$
v^\star(e_n)>\max_{\partial B_1}v-\epsilon.
$$
As $\epsilon>0$ is arbitrary, 
$$
v^\star(e_n)\ge \max_{\partial B_1}v.
$$
\end{proof}

\par The last fact we will need in order to prove Theorem \ref{CapThm} is the P\'olya-Szeg\"o inequality. It states for each admissible  $v\in {\cal D}^{1,p}(\R^n)$, the inequality
\be\label{PolyaCap}
\int_{\R^n}\left|Dv^\star\right|^pdx\le \int_{\R^n}|Dv|^pdx
\ee
holds. This inequality and various other properties of cap and Steiner symmetrizations were verified by Van Schaftingen in \cite{VanS}. 

\begin{proof}[Proof of Theorem \ref{CapThm}]
Define 
$$
v(x):=\max\{u(x),0\},\quad x\in \R^n,
$$
where $u$ is the Morrey extremal which satisfies \eqref{HolderTwoPoint} and \eqref{upmone}.  It is easy to check that $v\in {\cal D}^{1,p}(\R^n)$; and $v$ is admissible since $\{u\ge c\}$ is compact for $c\in (0,1]$. We recall that $u(e_n)=\sup_{\R^n}u=1$, and in view of Lemma \ref{capMaxLemma}, 
$$
v^\star(e_n)=\max_{\partial B_1}u=1.
$$
Inequality \eqref{PolyaCap} also implies 
$$
\int_{\R^n}\left|Dv^\star\right|^pdx\le \int_{\R^n}\left|Dv\right|^pdx=\int_{\{x_n>0\}}|Du|^pdx.
$$

\par By definition, 
$v^\star\ge 0$, and we also have $\{u>c\}\subset \{x_n>0\}$ for $c>0$. As a result,
$$
\{v^\star>c\}=\{v>c\}^\star=\{u>c\}^\star\subset \{x_n>0\}
$$
by \eqref{HalfSpaceProp}. It follows that $v^\star|_{\{x_n\le 0\}}= 0$. Consequently, 
$$
\int_{\{x_n>0\}}\left|Dv^\star\right|^pdx\le \int_{\{x_n>0\}}|Du|^pdx
$$
and
\be\label{consistencyVstar}
v^\star|_{\{x_n=0\}} = 0.
\ee

\par Define $w$ as the odd extension of $v^\star|_{\{x_n\ge 0\}}$ to the half space $x_n<0$. That is, we set 
$$
w(x)=
\begin{cases}
v^\star(x), \quad & x_n\ge 0\\
-v^\star(x-2x_ne_n), \quad & x_n<0.
\end{cases}
$$
Using \eqref{consistencyVstar}, it is straightforward to verify $w\in {\cal D}^{1,p}(\R^n)$. Also note  
$$
w(e_n)=v^\star(e_n)=1\;\;\text{and}\;\; w(-e_n)=-v^\star(e_n)=-1
$$
and 
$$
\int_{\R^n}\left|Dw\right|^pdx=2\int_{\{x_n>0\}}\left|Dv^\star\right|^pdx
\le 2\int_{\{x_n>0\}}|Du|^pdx=\int_{\R^n}|Du|^pdx.
$$

\par We can then employ Lemma \ref{FunLemma} to conclude $u\equiv w$. In particular, 
$$
u|_{\{x_n\ge 0\}}=v^\star|_{\{x_n\ge 0\}}
$$
By Lemma \ref{capDecreaseLemma}, we also have that if $x^1,x^2\in \R^n$ satisfies 
\be
 |x^1|= |x^2|\;\; \text{and}\;\;0\le x^1_n\le x^2_n,
\ee
then
\be\label{SphericalOrder}
u(x^1)\le u(x^2).
\ee

\par Moreover, if 
\be
 |x^1|= |x^2|\;\; \text{and}\;\; x^1_n\le x^2_n\le 0,
\ee
then
$$
u(x^2-2x^2_ne_n)\le u(x^1-2x^1_ne_n).
$$
by our remarks above.  Since $u$ is antisymmetric, 
$$
-u(x^2)\le -u(x^1)
$$
which of course again gives \eqref{SphericalOrder}. Since
$$
u(x^1) \leq 0 \leq u(x^2)
$$
when
$$
x_n^1 \leq 0 \leq x_n^2,
$$
we conclude that \eqref{SphericalOrder} holds for all $x^1,x^2\in \R^n$ satisfying
$$
|x^1| = |x^2| \quad \text{and} \quad x_n^1 \leq x_n^2.
$$
\end{proof}

\appendix

\section{Approximation}
This section is devoted to showing that smooth functions are ``dense" in $\mathcal{D}^{1,p} (\R^n)$.  In the following proposition, we will say that a sequence $(u_k)_{k\in \N}\subset \mathcal{D}^{1,p} (\R^n)$ converges to $u$ in $C^{1-n/p}(\R^n)$ if 
\[
\lim_{k\rightarrow\infty}\left(\|u_k-u\|_\infty+[u_k-u]_{1-n/p}\right)= 0.
\]
Note that each $u_k$ or $u$ need not be bounded on $\R^n$. 
\begin{prop}\label{DensityThm}
For each $u\in  \mathcal{D}^{1,p} (\R^n)$, there is $(u_k)_{k\in \N}\subset C^\infty (\R^n) \cap \mathcal{D}^{1,p} (\R^n)$
such that 
$$
u_k\rightarrow u
$$
in $C^{1-n/p}(\R^n)$ and 
$$
Du_k\rightarrow Du
$$
in $L^p(\R^n;\R^n)$ as $k\rightarrow\infty$.
\end{prop}

\begin{proof}
Let $\eta \in C_c^\infty (\R^n)$ be a standard mollifier. That is, $\eta$ is nonnegative with $\text{supp}(\eta)\subset B_1$ and $\int_{\R^n}\eta d x=1$.  We set $\eta^\epsilon (x) := \epsilon^{-n} \eta ( x / \epsilon )$ and define the mollification of $u$ as $u^\epsilon := \eta^\epsilon * u$. That is,
\[
u^\epsilon(x)=\int_{\R^n}\eta^\epsilon(y)u(x-y)dy=\int_{\R^n}\eta^\epsilon(x-y)u(y)dy \,, \quad x\in \R^n. 
\]
Consequently, $u^\epsilon\in C^\infty(\R^n)$. Also observe
\begin{eqnarray}
| u^\epsilon(x) - u(x)| &=& \left| \int_{\R^n} \eta^\epsilon (y) ( u(x-y) - u(x)) d y \right| \nonumber \\
&\leq& \int_{\R^n} \eta^\epsilon (y) | u(x-y) - u(x) | d y \nonumber \\
&\leq& [u]_{1-n/p} \int_{\R^n}\eta^\epsilon (y) |y|^{1-n/p} d y \nonumber \\
&=& [u]_{1-n/p} \int_{B_\epsilon(0)}\eta^\epsilon (y) |y|^{1-n/p} d y \nonumber \\
&\leq& [u]_{1-n/p} \epsilon^{1-n/p} 
\end{eqnarray}
for $x\in \R^n$. Thus, $u^\epsilon \rightarrow u$ uniformly in $\R^n$.

\par We also have that
\begin{equation}\label{uepsGradientFormula}
D u^\epsilon(x) = \int_{\R^n} \eta^\epsilon (x-y) Du (y) d y
\end{equation}
(Theorem 1 in section 5.3 of \cite{Ev}). And since $Du \in L^p (\R^n;\R^n)$, 
\[
Du^\epsilon \to Du\;\;\text{in}\;\; L^p (\R^n;\R^n) 
\]
as $\epsilon \to 0^+$ (Theorem 4.22 in section 4.4 of \cite{Br}).  In addition, we can invoke Morrey's inequality to find 
\[
[u^\epsilon-u]_{1-n/p} \leq C_*\left(\int_{\R^n}|Du^\epsilon-Du|^{p}dx\right)^{1/p}\rightarrow 0
\]
as $\epsilon \rightarrow 0^+$.  As a result, 
$$
u^\epsilon \rightarrow u\;\;\text{in}\;\; C^{1-n/p}(\R^n)
$$
as $\epsilon \to 0^+$, as well.

\par Suppose $\epsilon_k>0$ with $\lim_{k\rightarrow\infty}\epsilon_k=0$ and set $u_k:=u^{\epsilon_k}$ for $k\in \N$. We then have that $(u_k)_{k\in \N}\subset C^\infty (\R^n) \cap \mathcal{D}^{1,p} (\R^n)$ satisfies the desired conclusions. 
\end{proof}


\typeout{get arXiv to do 4 passes: Label(s) may have changed. Rerun}

\end{document}